\documentclass{amsart}
\title[On the interpolation constant in Orlicz spaces]
{On the interpolation constant for
subadditive operators in Orlicz spaces}

\author[A. Yu. Karlovich]{Alexei Yu. Karlovich}
\address{Departamento de Matem\'{a}tica,
Insituto Superior T\'{e}cnico, Av. Rovisco Pais 1,
1049-001 Lisbon, Portugal}
\email{akarlov@math.ist.utl.pt}

\author[L. Maligranda]{Lech Maligranda}
\address{Department of Mathematics, Lule{\aa} University of
Technology, SE\mbox-971 87 Lule{\aa}, Sweden}
\email{lech@sm.luth.se}

\thanks{The first author is supported by F.C.T. (Portugal) grants
SFRH/BPD/11619/2002 and FCT/FEDER/POCTI/MAT/59972/2004.}

\newtheorem{theorem}{Theorem}
\newtheorem{proposition}[theorem]{Proposition}
\newtheorem{lemma}[theorem]{Lemma}
\newtheorem{example}[theorem]{Example}
\newtheorem{remark}[theorem]{Remark}


\begin{document}
\maketitle
\begin{abstract}
Let $1\le p<q\le\infty$ and let $T$ be a subadditive operator acting on
$L^p$ and $L^q$. We prove that $T$ is bounded on the Orlicz space
$L^\varphi$, where $\varphi^{-1}(u)=u^{1/p}\rho(u^{1/q-1/p})$ for
some concave function $\rho$ and
\[
\|T\|_{L^\varphi\to L^\varphi}\le C\max\{\|T\|_{L^p\to
L^p},\|T\|_{L^q\to L^q}\}.
\]
The interpolation constant $C$, in general, is less than $4$ and, in many
cases, we can give much better estimates for $C$. In particular, if $p=1$
and $q=\infty$, then the classical Orlicz interpolation theorem holds for
subadditive operators with the interpolation constant $C=1$. These results
generalize our results for linear operators obtained in \cite{KM01}.
\end{abstract}

\section{Introduction}
Let $(\Omega,\Sigma,\mu)$ be a complete $\sigma$-finite measure space.
A Banach lattice $X$ on $(\Omega,\Sigma,\mu)$ is a Banach space of
(equivalence classes of $\mu$-a.e. equal) real or complex-valued functions
on $\Omega$ such that if $|x(t)|\le|y(t)|$ $\mu$-a.e. where $y\in X$ and
$x$ is $\mu$-measurable, then $x\in X$ and $\|x\|_X\le \|y\|_X$. Lebesgue
and Orlicz spaces are examples of Banach lattices.

Let $\varphi:[0, \infty) \to [0, \infty]$ be
an {\it Orlicz function}, that is, a nondecreasing convex function such
that $\varphi(0) = 0$ and $\lim\limits_{u\to 0+}\varphi(u) = 0$ but not
identically zero or infinity on $(0,\infty)$. For a measurable real or
complex-valued function $x$, define a functional ({\it modular})
\begin{equation}\label{eq:modular}
I_\varphi(x)
:=
\int_\Omega\varphi(|x(t)|) \,d\mu(t)
=
\int_0^\infty\varphi(x^*(s)) \,ds,
\end{equation}
where $x^*$ is the non-increasing rearrangement of the function $x$. The second
equality in (\ref{eq:modular}) means that the modular is
rearrangement invariant.
The {\it Orlicz space} $L^\varphi=L^\varphi(\Omega,\Sigma,\mu)$ is the set of
all equivalence classes of $\mu$-measurable functions on $\Omega$ such that
$I_\varphi(\lambda x) < \infty$ for some $\lambda = \lambda(x) > 0 $.
This space is a Banach space with two norms: the {\it Luxemburg-Nakano norm}
\[
\|x\|_\varphi:=\inf\left\{\lambda>0\ : \ I_\varphi(x/\lambda)\leq 1
\right\},
\]
and the {\it Orlicz norm} (in the Amemiya form)
\[
\|x\|_\varphi^0:=\inf_{k>0}\frac{1}{k} [1+I_\varphi(kx)].
\]
Since
$\|x\|_\varphi = \inf\limits_{k>0}\frac{1}{k}\max\big\{1, I_\varphi(kx)\big\}$
it follows that
\[
\|x\|_\varphi\le \|x\|_\varphi^0\le 2\|x\|_\varphi.
\]
The Orlicz space $L^\varphi$ equipped with each of the above two norms is a
rearrangement-invariant space, sometimes also called a symmetric
space with the Fatou property. For general properties of Orlicz spaces
we refer to the books \cite{BS88,KR61,LT79,M89}.

It is well known that any two Banach lattices $X_0$ and $X_1$
on the same measure space $(\Omega, \Sigma, \mu)$ forms a Banach couple
$(X_0, X_1)$ in the sense of interpolation theory (see \cite[p.~42]{KPS82}).
The intersection $X_0\cap X_1$ and the sum $X_0+X_1$ of these two
spaces are also Banach lattices on $(\Omega,\Sigma,\mu)$ with the
standard norms
(cf. \cite{BS88,BL76,KPS82}). An operator $T$ mapping $X_0+X_1$ into itself
is said to be {\it subadditive} if for all $x,y\in X_0+X_1$,
\[
|T(x+y)(t)| \le |Tx(t)|+|Ty(t)|
\quad\mu\mbox{-a.e. on}\ \Omega.
\]
If, in addition, we have also that
$|T(\lambda x)(t)| = |\lambda|\,|Tx(t)|$ $\mu$-a.e. on $\Omega$ for
any $x\in X_0+X_1$
and any scalar $\lambda$, then the operator $T$ is called {\it sublinear}.

By $\mathcal{A}(X_0,X_1)$ we denote the class of all admissible operators,
i.e., subadditive operators $T:X_0+X_1\to X_0+X_1$ such that the restrictions
$T|_{X_i}: X_i\to X_i$ are bounded for $i=1,2$. Put
\[
M_i
:=
\|T|_{X_i}\|_{X_i\to X_i}
=
\sup_{x\in X_i, x\ne 0}\frac{\|Tx\|_{X_i}}{\|x\|_{X_i}},
\quad
M:=\max\{M_0,M_1\}.
\]
A Banach space $X$ is said to be an {\it interpolation space for subadditive
operators} between Banach lattices $X_0$ and $X_1$ on the same
measure space $(\Omega,\Sigma,\mu)$ if
\[
X_0\cap X_1\subset X\subset X_0+X_1
\]
with continuous embeddings, every operator $T\in\mathcal{A}(X_0,X_1)$
maps $X$ into itself, and
\[
\|T\|_{X\to X}\le CM.
\]
The constant $C$ is called an {\it interpolation constant}.

A typical example of a subadditive operator which is not linear
is the Hardy-Littlewood maximal operator. It is well known that
this operator is bounded on all Lebesgue spaces $L^p(\mathbb{R}^n)$,
$1<p\le\infty$. Its operator norm on $L^p(\mathbb{R})$, $1<p<\infty$,
was calculated only recently by Grafakos and Mongomery-Smith \cite{GMS97}
(for dimensions $n\ge 2$, the problem is still open).
A natural question about generalizations of those results to
more general, for instance, Orlicz spaces, arises. This question
was our particular motivation for the present work.

The aim of this paper is to study the interpolation of subadditive
operators from the couple of Lebesgue spaces $L^p$ and $L^q$ to
an Orlicz space with the special attention to the interpolation
constant $C$. The corresponding problem for linear operators
was considered in our paper \cite{KM01}.
The interpolation of sublinear and subadditive operators between general
Banach lattices and for the Lions-Peetre real K-method and the
Calder{\'o}n complex method was considered in \cite{M89-CM}.
Despite a vast number
of works on the interpolation of Orlicz spaces and their generalizations,
only a few of them took care upon good estimates for the interpolation
constant. We will not go into details here but refer to historical remarks
in \cite{KM01} and \cite[Chapter~14]{M89}.

A function $\rho:[0,\infty)\to[0,\infty)$ is said to be {\it quasi-concave}
if it is continuous and positive on $\mathbb{R}_+:=(0,\infty)$ and
\[
\rho(s) \le \max (1,s/t) \rho(t)
\quad\mbox{for all}\quad s,t>0.
\]
Let $\mathcal{P}$ be the set of all quasi-concave functions and let
$\widetilde{\mathcal{P}}$ denote the subset of all concave functions
in $\mathcal{P}$. Note that if $\rho\in\mathcal{P}$, then its {\it concave
majorant} $\widetilde{\rho}$ defined by
\[
\widetilde{\rho}(t):=\inf_{s>0}\left(1+\frac{t}{s}\right)\rho(s)
\]
belongs to $\widetilde{\mathcal{P}}$ and
\begin{equation}\label{eq:conc-maj}
\rho(t)\le \widetilde{\rho}(t) \le 2\rho(t)
\quad\mbox{for all}\quad t>0.
\end{equation}
The constant $2$ is best possible.

Clearly, if $\theta\in(0,1)$, then
$\rho(t)=t^\theta$ belongs to $\widetilde{\mathcal{P}}$. Let us give a
nontrivial example of a function in $\widetilde{\mathcal{P}}$.
For $0 < \theta < 1$ and $a, b \in \mathbb{R}$, let $\rho(0) = 0$ and
$\rho(t) = t^{\theta}[\ln (e + t)]^{a} [\ln (e + 1/t)]^{b}$
for $t>0$. Then $\rho \in \widetilde{\mathcal{P}}$.

Following Gustavsson and Peetre
\cite{GP77} (see also \cite[Chap.~14]{M89}
and \cite{KM01}), we suppose that
\[
\varphi^{-1}(u)=u^{1/p}\rho(u^{1/q-1/p})
\quad\mbox{for all}\quad
u>0,
\]
where $1\le p<q\le\infty$ and $\rho\in\widetilde{\mathcal{P}}$. In this case
$\varphi$ is convex and the (well defined) Orlicz space $L^\varphi$
is an intermediate space between $L^p$ and $L^q$, that is,
\[
L^p\cap L^q\subset L^\varphi\subset L^p+L^q
\]
with both embeddings being continuous (see, e.g. \cite[Lemma~14.2]{M89}).

The paper is organized as follows. Section~\ref{sec:Peetre} contains
some information on the Peetre $L$-functional defined for a couple of
Banach lattices. In Section~\ref{sec:interpolation-limiting}, we study
the limiting case of the interpolation between $L^p$ and $L^\infty$.
The proof of the interpolation theorem is based on the Kr\'ee formula
and the Hardy-Littlewood-P\'olya majorization theorem.
In Section~\ref{sec:interpolation-generic}, we embark on the generic
case of the interpolation between $L^p$ and $L^q$ whenever $1\le p<q<\infty$.
Our approach goes back to Peetre \cite{Peetre70}. The sharp estimate
for the modified $L$-functional of the couple $(L^p,L^q)$ due to
Sparr \cite{Sparr78} is the main ingredient of our proof. For completeness
we also formulate a known interpolation theorem (see \cite{KM01})
for linear operators in Section~\ref{sec:interpolation-linear}. It gives
a slightly better estimate for the interpolation constant in the case
$1<p<q<\infty$ because in this case one can employ duality arguments.
\section{Peetre $L$-functional on Banach lattices}
\label{sec:Peetre}
The following Peetre $L$-functional plays a significant role in the real
method of interpolation theory (see \cite{Peetre70} and also
\cite{BRR91, Sparr78}). It is defined for $0<p,q,t<\infty$ and
for $x\in X_0+X_1$ (where $X_0$ and $X_1$ are arbitrary Banach spaces,
not necessarily having a lattice structure) by
\[
K_{p,q}(t,x;X_0,X_1):=
\inf\big\{\|x_0\|_{X_0}^p + t\|x_1\|_{X_1}^q :\
x = x_0+x_1,\
x_0 \in X_0,\
x_1\in X_1\big\}.
\]
In the case when $p = q = 1$ this $L$-functional is the classical
Peetre $K$-functional, which we shortly denote by $K(t,x;X_0,X_1)$.
\begin{proposition}\label{pr:functional-subadditive}
Let $0<p,q<\infty$ and let $X_0,X_1$ be {\rm(}real or complex{\rm\,)}
Banach lattices on a complete
$\sigma$-finite measure space $(\Omega,\Sigma,\mu)$. If $t>0$ and
$x\in X_0+X_1$,
then $K_{p,q}(t,x;X_0,X_1)$ is equal to
\[
\inf\big\{\|x_0\|_{X_0}^p+t\|x_1\|_{X_1}^q:\
|x|\le x_0+x_1,\
0 \leq x_0\in X_0,\
0\le x_1\in X_1\big\}.
\]
\end{proposition}
\begin{proof}
If $x=x_0+x_1$ with $x_0\in X_0$ and $x_1\in X_1$, then
\[
|x|=x_0e^{-i\theta}+x_1e^{-i\theta}=y_0+y_1,
\]
where $\theta:\Omega\to\mathbb{R}$, and
\[
K_{p,q}(t,|x|;X_0,X_1)
\le
\|y_0\|_{X_0}^p+t\|y_1\|_{X_1}^q
=
\|x_0\|_{X_0}^p+t\|x_1\|_{X_1}^q.
\]
Hence
\begin{equation}\label{eq:fs-1}
K_{p,q}(t,|x|;X_0,X_1)\le K_{p,q}(t,x;X_0,X_1).
\end{equation}
Similarly, if $|x|=x_0+x_1$ with $x_0\in X_0$ and $x_1\in X_1$, then
\[
x=|x| e^{i\theta}= x_0e^{i\theta} +x_1e^{i\theta}=y_0+y_1,
\]
and
\[
K_{p,q}(t,x;X_0,X_1)
\le
\|y_0\|_{X_0}^p+t\|y_1\|_{X_1}^q
=
\|x_0\|_{X_0}^p+t\|x_1\|_{X_1}^q,
\]
from which we obtain the estimate
\begin{equation}\label{eq:fs-2}
K_{p,q}(t,x;X_0,X_1)\le K_{p,q}(t,|x|;X_0,X_1).
\end{equation}
Combining (\ref{eq:fs-1}) and (\ref{eq:fs-2}), we arrive at
\begin{equation}\label{eq:fs-3}
K_{p,q}(t,x;X_0,X_1)= K_{p,q}(t,|x|;X_0,X_1).
\end{equation}
Now let $|x|=x_0+x_1=\mathrm{Re}\,x_0+\mathrm{Re}\,x_1$,
where $x_0\in X_0$ and $x_1\in X_1$. Consider the sets
\begin{eqnarray*}
A_1&:=&
\big\{t\in\Omega:\ \mathrm{Re}\,x_0(t)\ge 0, \ \mathrm{Re}\,x_1(t)\ge 0\big\},
\\
A_2&:=&
\big\{t\in\Omega:\ \mathrm{Re}\,x_0(t)\ge 0, \ \mathrm{Re}\,x_1(t)< 0\big\},
\\
A_3&:=&
\big\{t\in\Omega:\ \mathrm{Re}\,x_0(t)< 0, \ \mathrm{Re}\,x_1(t)\ge 0\big\}.
\end{eqnarray*}
Put
\[
x_0'(t):=\left\{\begin{array}{lll}
\mathrm{Re}\,x_0(t) &\mbox{if} &t\in A_1,\\
\mathrm{Re}\, x_0(t)+\mathrm{Re}\, x_1(t) &\mbox{if} &t\in A_2,\\
0 &\mbox{if} & t\in\Omega\setminus(A_1\cup A_2),
\end{array}\right.
\]
\[
x_1'(t):=\left\{\begin{array}{lll}
\mathrm{Re}\,x_1(t) &\mbox{if} &t\in A_1,\\
\mathrm{Re}\, x_0(t)+\mathrm{Re}\, x_1(t) &\mbox{if} &t\in A_3,\\
0 &\mbox{if} & t\in\Omega\setminus(A_1\cup A_3).
\end{array}\right.
\]

Since $\Omega\setminus(A_1\cup A_2) = A_{3}$ and
$\Omega\setminus(A_1\cup A_3)=A_{2}$
(here we do not distinguish sets differing by a set of $\mu$-measure zero)
it follows that $|x|=x_0'+x_1'$ and
$0\le x_i'\le|\mathrm{Re}\, x_i|\le |x_i|$ for $i=0,1$.
Thus the sets
\begin{eqnarray*}
S_1 &:=&\big\{x\in X_0+X_1:\ |x|=x_0+x_1, \ x_0\in X_0, \ x_1\in X_1\big\},
\\
S_2 &:=&
\big\{x\in X_0+X_1:\ |x|=x_0+x_1, \ 0\le x_0\in X_0, \ 0\le x_1\in X_1\big\}
\end{eqnarray*}
coincide. If $x\in X_0+X_1$ is such that $|x|\le x_0+x_1$ with
$0\le x_0\in X_0$ and $0\le x_1\in X_1$, then for $i=0,1$ put
\[
y_i:=\left\{\begin{array}{lll}
\frac{x_i|x|}{x_0+x_1} &\mbox{if}& x_0+x_1>0,\\
0                      &\mbox{if}& x_0+x_1=0.
\end{array}\right.
\]
In that case $|x|=y_0+y_1$ and $0\le y_i\le x_i$. Hence the set
\[
S_3: =
\big\{x\in X_0+X_1:\ |x|\le x_0+x_1, \ 0\le x_0\in X_0, \ 0\le x_1\in X_1\big\}
\]
coincides with $S_2=S_1$. Thus
\begin{eqnarray}
\label{eq:fs-4}
K_{p,q}(t,|x|;X_0,X_1)
&=&
\inf_{x\in S_1}\big(\|x_0\|_{X_0}^p+t\|x_1\|_{X_1}^q\big)
\\
\nonumber
&=&
\inf_{x\in S_3}\big(\|x_0\|_{X_0}^p+t\|x_1\|_{X_1}^q\big).
\end{eqnarray}
We finish the proof combining (\ref{eq:fs-3}) and (\ref{eq:fs-4}).
\end{proof}

The above statement allows us to study admissible subadditive operators
on Banach lattices.
\begin{proposition}\label{pr:K-bound}
Let $0<p,q<\infty$ and let $X_0,X_1$ be {\rm(}real or complex{\rm\,)}
Banach lattices on a complete
$\sigma$-finite measure space $(\Omega,\Sigma,\mu)$. Suppose
$T\in{\mathcal{A}}(X_0,X_1)$ and $x\in X_0+X_1$. Then
\[
K_{p,q}\left(t,\frac{Tx}{M};X_0,X_1\right) \le K_{p,q}(t,x;X_0,X_1)
\quad\mbox{for all}\quad t>0.
\]
\end{proposition}
\begin{proof}
The proof is standard. If $x=x_0+x_1$ is any decomposition of $x\in X_0+X_1$
such that $x_0\in X_0$ and $x_1\in X_1$, then taking into account that $T$
is subadditive, we have
\[
\frac{|Tx|}{M}\le\frac{|Tx_0|}{M}+\frac{|Tx_1|}{M}.
\]
{From} Proposition~\ref{pr:functional-subadditive} it follows that
\begin{eqnarray*}
K_{p,q}\left(t,\frac{Tx}{M};X_0,X_1\right)
&\le&
\left\|\frac{|Tx_0|}{M}\right\|_{X_0}^p+\left\|\frac{|Tx_1|}{M}\right\|_{X_1}^q
\\
&\le&
\left(\frac{M_0}{M}\right)^p\|x_0\|_{X_0}^p+t\left(\frac{M_1}{M}\right)^q\|x_1\|_{X_1}^q
\\
&\le&
\|x_0\|_{X_0}^p+t\|x_1\|_{X_1}^q.
\end{eqnarray*}
Taking the infimum over all $x\in X_0+X_1$ such that $x_0\in X_0$ and
$x_1\in X_1$,
we arrive at the desired inequality.
\end{proof}
\section{Interpolation between $L^p$ and $L^\infty$}
\label{sec:interpolation-limiting}
Our first main result is the following interpolation theorem.
\begin{theorem}\label{th:int-c}
Suppose $1 \leq p < \infty$.
\begin{enumerate}
\item[{\rm (a)}]
If $\psi(u)=\varphi(u^{1/p})$ is a convex function and
$T\in\mathcal{A}(L^p,L^\infty)$, then
\[
I_\varphi\left(\frac{Tx}{2^{1-1/p}M}\right)
\le
I_\varphi(x)\quad\mbox{for all}\quad x\in L^p+ L^\infty.
\]

\item[{\rm (b)}] If $\varphi^{-1}(u) = u^{1/p}\rho(u^{-1/p})$ with
$\rho\in\widetilde{\mathcal{P}}$ such that
$\rho_*(\mathbb{R}_+)=\mathbb{R}_+$, where $\rho_*(t):=t\rho(1/t)$,
then the Orlicz space $L^\varphi$ {\rm(}with both, the Luxemburg-Nakano
and the Orlicz norm\,{\rm)} is an interpolation space for
subadditive operators between $L^p$ and $L^\infty$, and
\[
\|T\|_{L^\varphi\to L^\varphi}
\le C\max
\big\{
\|T\|_{L^p\to L^p}, \|T\|_{L^\infty\to L^\infty}
\big\}
\]
for any $T\in\mathcal{A}(L^p,L^\infty)$, where $C\le 2^{1-1/p}$.
\end{enumerate}
\end{theorem}
\begin{proof}  (a) The proof is developed by analogy with the proof
of \cite[Theorem~4.2(b)]{KM01}. For all $x\in L^p+L^\infty$ and $t>0$,
according to the Kr\'ee formula (see \cite[Theorem~5.2.1]{BL76}), we have
\begin{eqnarray}\label{eq:mod-infty-2}
&&
\left(\int_0^t x^*(s)^p ds \right)^{1/p}
\le K(t^{1/p}, x; L^p,L^\infty) \le 2^{1-1/p}
\left( \int_0^t x^*(s)^p ds \right)^{1/p}
\!\!\!.
\end{eqnarray}
Notice that the constant $2^{1-1/p}$ on the right-hand side is best possible
(see Bergh \cite{Bergh73}). Due to Proposition~\ref{pr:K-bound},
\begin{equation}\label{eq:mod-infty-3}
K\left(t,\frac{Tx}{M};L^p,L^\infty\right) \le K(t,x;L^p,L^\infty)
\quad\mbox{for all}\quad t>0.
\end{equation}
  {From} (\ref{eq:mod-infty-2}) and (\ref{eq:mod-infty-3}) we obtain that
\[
\int_0^t\left(\frac{(Tx)^*(s)}{2^{1-1/p}M}\right)^pds \le \int_0^t
x^*(s)^p\,ds
\quad\mbox{for all}\quad t>0.
\]
Since $\psi(u)=\varphi(u^{1/p})$ is convex it follows, by the
Hardy-Littlewood-P\'olya majorization theorem (see, e.g. \cite[p.~88]{BS88})
and equality (\ref{eq:modular}), that
\begin{eqnarray*}
\int_\Omega
\varphi\left(\frac{|Tx(t)|}{2^{1-1/p}M} \right)\,d\mu(t)
&=&
\int_0^\infty
\varphi\left(\frac{(Tx)^*(s)}{2^{1-1/p}M} \right)\,ds\\
&=&
\int_0^\infty \psi\left( \left[ \frac{(Tx)^*(s)}{2^{1-1/p}M}
\right]^p \right)\,ds \\
&\leq&
\int_0^\infty \psi(x^*(s)^p)\,ds\\
&=&
\int_0^\infty \varphi(x^*(s))\,ds \\
&=&
\int_{\Omega}
\varphi\left( |x(t)| \right)\,d\mu(t),
\end{eqnarray*}
and this is a desired statement. Part (a) is proved.

\medskip
\noindent
(b) It is possible to prove that if
$\varphi^{-1}(u) = u^{1/p}\rho(u^{-1/p})$ then the function
$\psi(u)=\varphi(u^{1/p})$ is convex (cf. \cite[Lemma 3.2(d)]{KM01} for
details). Hence, by part (a), we obtain the modular estimate
\[
I_\varphi\left(\frac{Tx}{2^{1-1/p}M}\right)\le I_\varphi(x)
\quad\mbox{for all}\quad x\in L^\varphi,
\]
which implies
\[
\|Tx\|_\varphi\le 2^{1-1/p}M\|x\|_\varphi \quad {\rm and} \quad
\|Tx\|^0_\varphi\le
2^{1-1/p}M\|x\|^0_\varphi
\]
for all $x\in L^\varphi$.
\end{proof}

Theorem~\ref{th:int-c} was proved for linear operators in our paper
\cite{KM01}. In the case $p=1$, Theorem~\ref{th:int-c}(b) generalizes
the well-known Orlicz interpolation theorem to subadditive operators.
Orlicz proved it in 1934 for linear operators and with certain constant $C>1$.
{From} the Calder\'on-Mitjagin interpolation theorem (see, e.g.
\cite[Chap.~2, Theorem~4.9]{KPS82}) it follows that, in fact, the
interpolation constant is equal to $1$. A simple proof of the Orlicz
interpolation theorem for linear and for Lipschitz operators with the
interpolation constant $1$, together with its applications, was given by
one of the authors \cite{M89-Studia} (see also \cite{M89}). Lorentz and
Shimogaki \cite[Theorem~7]{LS71} observed that if $1\le p<\infty$,
$\varphi(u)=\int_0^u(u-t)^p\,dm(t)$, where $m:\mathbb{R}_+\to\mathbb{R}_+$
is an increasing function, and $T\in\mathcal{A}(L^p,L^\infty)$ is a linear
operator, then
$\|T\|_{L^\varphi\to L^\varphi}\le\max\{\|T\|_{L^p},\|T\|_{L^\infty}\}$.
\section{Interpolation between the Lebesgue spaces $L^p$ and $L^q$\\ with
$1\le p<q<\infty$}\label{sec:interpolation-generic}
We will need the following properties of convex and concave functions.
\begin{lemma}\label{le:properties}
Suppose that $1\le p<q<\infty$ and, for some $\rho\in\widetilde{\mathcal{P}}$,
\[
\varphi^{-1}(u)=u^{1/p}\rho(u^{1/q-1/p}) \quad\mbox{for all}\quad
u>0.
\]
Then $\varphi$ is convex and there exists a function $h\in\mathcal{P}$
such that
\[
\varphi(u)=u^q h(u^{p-q})\quad\mbox{for all}\quad u>0.
\]
\end{lemma}
\begin{proof}
For a proof, see \cite[Lemma~3.2(b)]{KM01}.
\end{proof}

Note that the above lemma guarantees only that $h\in\mathcal{P}$
and $h$ need not necessarily to be concave. We illustrate this observation
with the following simple example.
\begin{example}{\rm
If $1\le p<q<\infty$ and $\varphi^{-1}(u)=u^{1/p}\rho(u^{1/q-1/p})$ with
$\rho(t)=\min\{1,t\}$, then $\varphi(u)=u^qh(u^{p-q})$ with
$h(t)=\max\{1,t\}$. Obviously, $\rho\in\widetilde{\mathcal{P}}$
and $h\in\mathcal{P}\setminus\widetilde{\mathcal{P}}$.}
\end{example}
\begin{lemma}[Peetre, 1966]
\label{le:concave-represent}
Every function $h\in\widetilde{\mathcal{P}}$ can be represented in the form
\begin{equation}\label{eq:concave-represent-1}
h(u)=a_h+b_hu+\int_0^\infty\min\{u,t\}\,dm(t)
\quad\mbox{for all}\quad u>0,
\end{equation}
where
\begin{equation}\label{eq:concave-represent-2}
a_h:=\lim_{u\to 0+} h(u),
\quad
b_h:=\lim_{u\to \infty}\frac{h(u)}{u},
\end{equation}
and $m:\mathbb{R}_+\to\mathbb{R}_+$ is a nondecreasing function
{\rm(}in fact, $m(t)=-h'(t)${\rm)}.
\end{lemma}
\begin{proof}
A proof of this result is contained in \cite{Peetre66},
\cite[Lemma~5.4.3]{BL76}.
\end{proof}

We consider the modified Peetre $L$-functional $K^*_{p,q}$ for the
couple of Lebesgue spaces $(L^p,L^q)$ defined by
\[
K^*_{p,q}(t,x;L^p,L^q):= \int_\Omega
\min\big\{|x(s)|^p,t|x(s)|^q\big\}\,d\mu(s).
\]
\begin{lemma}[Sparr, 1978]
\label{le:sparr}
Suppose $1\le p <q<\infty$. If $x,y\in L^p+L^q$ and
\[
K_{p,q}(t,x;L^p,L^q)\le
K_{p,q}(t,y;L^p,L^q)\quad \mbox{for all}\quad t>0,
\]
then
\[
K^*_{p,q}(t,x;L^p,L^q)\le
\gamma_{p,q}
K^*_{p,q}(t,y;L^p,L^q)\quad\mbox{for all}\quad t>0,
\]
where
\[
\gamma_{p,q}:=\inf \Bigg\{\gamma>0\::\:
\inf_{\tiny\begin{array}{c}
x+y=\gamma, \\
x,y\ge 0
\end{array}}
(x^p+y^q )=1 \Bigg\}.
\]
The constant $\gamma_{p,q}$ cannot be replaced by any smaller constant.
\end{lemma}
\begin{proof}
For a proof, see \cite[Lemma~5.1 and Example~5.3]{Sparr78}.
\end{proof}

We have found natural to attribute Sparr's name to the constants
$\gamma_{p,q}$.
Sparr observed that $1<\gamma_{p,q}<2$. The Sparr constants play an important
role in our final result. Now we give some more precise information about the
Sparr constants.
\begin{proposition}[Karlovich-Maligranda, 2001]
\label{pr:sparr-constant}
Let $1\le p,q<\infty$.
\begin{enumerate}
\item[{\rm(a)}] We have $\gamma_{p,q}=\gamma_{q,p}$ and $\gamma_{1,1}=1$.

\item[{\rm(b)}] If $q>1$, then
\[
\gamma_{p,q}=\inf \left\{
x+ \left(\frac{p}{q}x^{p-1}\right)^{1/(q-1)}\ : \
x^p+ \left(\frac{p}{q}x^{p-1} \right)^{q/(q-1)} = 1 \right\}.
\]
In particular, $\gamma_{q,q}=2^{1-1/q}$ and
$\gamma_{1,q}=1+q^{1/(1-q)}-q^{q/(1-q)}$.

\item[{\rm(c)}] $\gamma_{p,q}$ continuously increases in $p$ and $q$.

\item[{\rm(d)}] If $p\le q$, then $2^{1-1/p}\le\gamma_{p,q}\le 2^{1-1/q}$.
\end{enumerate}
\end{proposition}
\begin{proof}
A proof can be found in \cite[Proposition~4.3]{KM01}.
\end{proof}

We are ready to prove our second main result: the modular estimate and the
estimate for the norm of an admissible operator $T\in\mathcal{A}(L^p,L^q)$.
\begin{theorem}\label{th:int-conc}
Suppose $1\le p<q<\infty$.
\begin{enumerate}
\item[{\rm (a)}] If $\varphi(u)=u^q h(u^{p-q})$ for some
$h\in\widetilde{\mathcal{P}}$
and $T\in\mathcal{A}(L^p,L^q)$, then
\[
I_\varphi\left(\frac{Tx}{M}\right)
\le
\gamma_{p,q}I_\varphi(x)
\quad\mbox{for all}\quad x\in L^p\cap L^q.
\]

\item[{\rm (b)}] If $\varphi^{-1}(u)=u^{1/p}\rho(u^{1/q-1/p})$ for
some $\rho\in\widetilde{\mathcal{P}}$. Then the Orlicz space
$L^\varphi$ {\rm(}with both, the Luxemburg and the Orlicz norm\,{\rm)} is an
interpolation space for subadditive operators between $L^p$ and $L^q$,
and
\[
\|T\|_{L^\varphi\to L^\varphi}
\le C\max
\big\{
\|T\|_{L^p\to L^p},
\|T\|_{L^q\to L^q}
\big\}
\]
for any $T\in\mathcal{A}(L^p,L^q)$, where
$C\le
(2\gamma_{p,q})^{1/p}
\le
2^{(2-1/q)/p}<4$.
\end{enumerate}
\end{theorem}
\begin{proof}
(a) The idea of the proof goes back to Peetre \cite{Peetre70}.
We follow the proof of \cite[Theorem~4.2(a)]{KM01}.
Due to Lemma~\ref{le:concave-represent}, the function $h$ can
be represented in the form (\ref{eq:concave-represent-1}). Hence,
\begin{equation}\label{eq:int-conc-0}
\varphi(u)=u^qh(u^{p-q})=a_h u^q+b_h u^p+
\int_0^\infty\min\{u^p,tu^q\}\, dm(t)
\end{equation}
for all $u\in\mathbb{R}_+$. Consequently,
\begin{eqnarray*}
I_\varphi\left(\frac{Tx}{M}\right)
&=&
\int_\Omega\varphi\left(\frac{|Tx(s)|}{M}\right)\,d\mu(s)
\\
&=&
a_h\left\|\frac{Tx}{M}\right\|_q^q+
b_h\left\|\frac{Tx}{M}\right\|_p^p
\nonumber\\
&&
+ \int_\Omega\left[
\int_0^\infty\min\left\{
\left(\frac{|Tx(s)|}{M}\right)^p,t\left(\frac{|Tx(s)|}{M}\right)^q
\right\}dm(t)
\right]d\mu(s).
\end{eqnarray*}
Since the operator $T$ is bounded in $L^p$ and $L^q$, we get
\begin{eqnarray*}
a_h\left\|\frac{Tx}{M}\right\|_q^q+b_h\left\|\frac{Tx}{M}\right\|_p^p
&\le&
a_h\left(\frac{M_1}{M}\right)^q \|x\|_q^q +
b_h\left(\frac{M_0}{M}\right)^p \|x\|_p^p
\\
&\le&
a_h\|x\|_q^q
+b_h\|x\|_p^p
\end{eqnarray*}
and, according to Proposition~\ref{pr:K-bound}, \[
K_{p,q}\left(t,\frac{Tx}{M};L^p,L^q\right) \le K_{p,q}(t,x;L^p,L^q)
\quad\mbox{for all}\quad t>0.
\]
Applying Sparr's Lemma~\ref{le:sparr} we obtain
\[
K^*_{p,q}\left(t,\frac{Tx}{M};L^p,L^q\right) \le \gamma_{p,q}
K^*_{p,q}(t,x;L^p,L^q)
\quad\mbox{for all}\quad t>0.
\]
Hence, by the Fubini theorem and in view of the definition of
$K^*_{p,q}$, we conclude that
\begin{eqnarray*}
&&
\int_{\Omega}\left[\int_0^\infty
\min\left\{
\left(\frac{|Tx(s)|}{M}\right)^p,t\left(\frac{|Tx(s)|}{M}\right)^q
\right\}
dm(t)\right]
d\mu(s)
\nonumber\\
&&
=
\int_0^\infty K^*_{p,q}\left(t,\frac{Tx}{M};L^p,L^q\right)dm(t)
\nonumber\\
&&
\le
\gamma_{p,q}
\int_0^\infty
K^*_{p,q}(t,x;L^p,L^q)\, dm(t)
\nonumber\\
&&
=
\gamma_{p,q}\int_{\Omega}
\left[
\int_0^\infty\min\big\{|x(s)|^p,t|x(s)|^q\big\}dm(t) \right] d\mu(s).
\label{eq:int-conc-4}
\end{eqnarray*}
Combining the above estimates and taking into account that $\gamma_{p,q}> 1$
we obtain
\begin{eqnarray*}
I_\varphi\left(\frac{Tx}{M}\right)
&\le&
a_h\|x\|_q^q +b_h\|x\|_p^p
\\
&&+
\gamma_{p,q}
\int_\Omega\left[
\int_0^\infty\min\big\{|x(s)|^p,t|x(s)|^q\big\}\,dm(t)\right]d\mu(s) \\ &\le&
\gamma_{p,q}\Bigg(
a_h\|x\|_q^q +b_h\|x\|_p^p
\\
&&+
\int_\Omega\left[
\int_0^\infty\min\big\{|x(s)|^p,t|x(s)|^q\big\}\,dm(t)\right]d\mu(s) \Bigg) \\
&=&
\gamma_{p,q}\int_\Omega\varphi(|x(s)|)\,d\mu(s)
\\
&=&\gamma_{p,q}I_\varphi(x).
\end{eqnarray*}
Part (a) is proved.
\medskip

\noindent
(b) This statement is proved by analogy with \cite[Theorem~5.1~(a)-(b)]{KM01}.
{From} \cite[Lemma~14.2]{M89} it follows that the function $\varphi$ is convex.
Hence the Orlicz space $L^\varphi$ is well defined.

By Lemma~\ref{le:properties}, there is a function $h\in\mathcal{P}$ such that
$\varphi(u)=u^qh(u^{p-q})$. {From} (\ref{eq:conc-maj}) we see that
$\widetilde{h}\in\widetilde{\mathcal{P}}$ and
\begin{equation}\label{eq:int-1}
\varphi(u)
\le
u^q\widetilde{h}(u^{p-q})
\le 2\varphi(u)
\quad\mbox{for all}\quad u>0.
\end{equation}
Applying Theorem~\ref{th:int-conc}(a) to the function
$\psi(u)=u^q \widetilde{h}(u^{p-q})$ and taking into account (\ref{eq:int-1}),
we obtain
\begin{equation}\label{eq:int-2}
I_\varphi\left(\frac{Tx}{M}\right)
\le
I_\psi\left(\frac{Tx}{M}\right)
\le
\gamma_{p,q}I_\psi(x)
\le
2
\gamma_{p,q}I_\varphi(x)
\end{equation}
for all $x\in L^p\cap L^q$. {From} the properties of $h$ one can conclude
that $\varphi$ satisfies the $\Delta_2$-condition on $[0,\infty)$ and
\[
I_\varphi\left(\frac{Tx}{(2\gamma_{p,q})^{1/p}M}\right) \le
\frac{1}{2\gamma_{p,q}}
I_\varphi\left(\frac{Tx}{M}\right)
\le I_\varphi(x)
\]
for all $x\in L^p\cap L^q$. Hence
\begin{equation}\label{eq:int***}
\|Tx\|_\varphi\le (2\gamma_{p,q})^{1/p}M\|x\|_\varphi,
\quad
\|Tx\|^0_\varphi\le(2\gamma_{p,q})^{1/p}M\|x\|^0_\varphi
\end{equation}
for all $x\in L^p\cap L^q$. Since $\varphi$ satisfies the $\Delta_2$-condition
for all $u\ge 0$, it follows that $L^p\cap L^q$ is dense in $L^\varphi$
(for the case of Orlicz spaces generated by $N$-functions and defined on
Euclidean spaces of finite measure, see \cite[Chap.~2]{KR61}; the proof in
a more general situation considered in this paper is analogous). Thus,
inequalities (\ref{eq:int***}) are fulfilled for all $x\in L^\varphi$.
This fact and Proposition~\ref{pr:sparr-constant}(d) show that
$C\le (2\gamma_{p,q})^{1/p}\le 2^{(2-1/q)/p}< 4$.
\end{proof}
\begin{remark}{\rm
If $1\le p<q<\infty$ and $\varphi(u) = u^qh(u^{p-q})$
with $h\in\widetilde{\mathcal{P}}$, then from the proof of the above
theorem it follows that $L^\varphi$ is an interpolation space
between $L^p$ and $L^q$, and we have a better estimate of the
interpolation constant:}
\[
C\le (\gamma_{p,q})^{1/p}\le 2^{1/(q'p)}<2.
\]
\end{remark}
As it was pointed out by Masty{\l}o \cite{Mastylo01}, Example 5.4 in our paper
\cite{KM01} illustrating this possibility is erroneous. We substitute it
by the following.
\begin{example}{\rm
The function $h$ given by $h(0)=0$ and $h(t)=t\ln(1+1/t)$ for $t>0$ belongs to
$\widetilde{\mathcal{P}}$ and for all $p,q$ satisfying $1\le p<q<\infty$ the
function
\[
\varphi(u)=u^qh(u^{p-q})=u^p\ln(1+u^{q-p})
\]
is convex on $[0,\infty)$.

Indeed, for $t > 0$ we have
\[
h'(t)= \ln\left(1+\frac{1}{t}\right) - \frac{1}{1+t} > 0,\quad
h''(t)=\frac{-1}{t(1+t)^{2}} < 0,
\]
thus $h$ is increasing and
concave on $\mathbb{R}_+$. On the other hand, simple calculations give
\begin{eqnarray*}
\varphi''(u)
&=&
p(p-1)u^{p-2}\ln(1+u^{q-p})+
p(q-p)\frac{u^{q-2}}{1+u^{q-p}}
\\[3mm]
&&+
(q-p)\frac{(q-1)u^{q-2}+(p-1)u^{2q-p-2}}{(1+u^{q-p})^2}.
\end{eqnarray*}
Since $p\ge 1$ and $q>p$, we have $\varphi''(u)\ge 0$ for all $u\ge 0$.
Thus $\varphi$ is convex on $[0,\infty)$.}
\end{example}
\section{Interpolation of linear operators}\label{sec:interpolation-linear}
The interpolation constant in Theorem~\ref{th:int-conc}(b) can be improved for
linear operators by using the duality argument. By $p'$ we denote the conjugate
number to $p$, $1< p<\infty$, defined by $1/p+1/p'=1$.
\begin{theorem}[Karlovich-Maligranda, 2001]
\label{th:int-lin}
Let $1<p<q<\infty$ and
\[
\varphi^{-1}(u)=u^{1/p}\rho(u^{1/q-1/p})
\]
for
some $ \rho\in\widetilde{\mathcal{P}}$. Then the Orlicz space
$L^\varphi$ {\rm(}with both, the Luxemburg-Nakano and the Orlicz norm\,{\rm)}
is an interpolation space for linear operators between $L^p$ and $L^q$, and
\[
\|T\|_{L^\varphi\to L^\varphi}
\le C\max \big\{\|T\|_{L^p\to L^p}, \|T\|_{L^q\to L^q}\big\}
\]
for any admissible linear operator $T\in\mathcal{A}(L^p,L^q)$,
where
\[
C\le
\min\big\{
(2\gamma_{p,q})^{1/p},(2\gamma_{q',p'})^{1/q'}
\big\}
\le
2^{1/(pq')+\min\{1/p,1/q'\}}<4.
\]
In particular, if either $1<p<q\le 2$ or $2\le p<q < \infty$, then
$C<2$.
\end{theorem}
\begin{proof}
This result is proved in \cite[Theorem~5.1(b)]{KM01}.
\end{proof}

The estimates we proved above can be used in the norm estimation of
some concrete linear operators (like Hardy operators, convolution
operators, integral operators, the Hilbert transform or other singular
integral operators) and subadditive operators (like maximal operators)
between Orlicz spaces.

\end{document}